\documentclass[a4paper, 11pt]{amsart}

\usepackage{amssymb,amsbsy,amsmath,amsfonts,amssymb,amscd}
\usepackage{latexsym}
\usepackage{amsxtra}
\usepackage{amscd}
\usepackage{graphics}

\usepackage{epic}
\input xy
\xyoption{all}

\usepackage{ mathrsfs, amsfonts}

\numberwithin{equation}{section}

\newcommand\sC{{\mathcal C}}

\newcommand\sW{{\mathcal W}}

\newcommand\sA{{\mathcal A}}

\newcommand\sL{{\mathcal L}}
\newcommand\sZ{{\mathcal Z}}
\newcommand\sB{{\mathcal B}}

\newcommand\sY{{\mathcal Y}}

\newcommand\pa{\partial}

\newcommand\s{\sigma}

\newcommand\Si{\Sigma}

\newcommand\fie{\varphi}
\DeclareMathOperator{\Pic}{Pic}

\newcommand{\CC}{\ensuremath{\mathbb{C}}}

\newcommand{\hol}{\ensuremath{\mathcal{O}}}

\newcommand{\PP}{\ensuremath{\mathbb{P}}}

\newcommand{\ra}{\ensuremath{\rightarrow}}

\def\eea{\end{eqnarray*}}
\def\bea{\begin{eqnarray*}}

\newcommand{\Proof}{{\it Proof. }}

\newcommand\dual{\mathrel{\raise3pt\hbox{$\underline{\mathrm{\thinspace d
\thinspace}}$}}}
\newcommand\qe{\ifhmode\unskip\nobreak\fi\quad $\Box$}       % box for QED

\def\BOX{\hfill\lower.5\baselineskip\hbox{$\Box$}}
\newcommand{\sR}{\ensuremath{\mathcal{R}}}

\usepackage{amsfonts,amsmath,mathrsfs,amssymb,amsthm,amscd,latexsym,amstext,amsxtra}

\usepackage[usenames]{color}
\newtheorem{theorem}{Theorem}[section]
\newtheorem{lemma}[theorem]{Lemma}
\newtheorem{corollary}[theorem]{Corollary}
\newtheorem{proposition}[theorem]{Proposition}
\newtheorem{remark}[theorem]{Remark}
\newtheorem{definition}[theorem]{Definition}
\newtheorem{example}[theorem]{Example}

\title [Deformation to hypersurface embeddding]{Deformation of a generically finite map to a hypersurface  embedding}
\author{Fabrizio Catanese and Yongnam Lee}
\address {Lehrstuhl Mathematik VIII\\
Mathematisches Institut der Universit\"at Bayreuth\\
NW II,  Universit\"atsstr. 30\\
95447 Bayreuth, Germany}
\email{fabrizio.catanese@uni-bayreuth.de}
\address {Department of Mathematical Sciences\\
KAIST\\
291, Daehak-ro, Yuseong-gu\\ 
Daejeon, 34141, Korea}
\email{ynlee@kaist.ac.kr}

\thanks{
The present work took place in the framework  of the 
 ERC Advanced grant n. 340258, `TADMICAMT'. Both authors would also like to 
 acknowledge the support and hospitality
 of KIAS, Seoul which they visited as Research Scholar, respectively as Affiliate Professor. The second author was also supported by Basic Science Research Program through the NRF of Korea (2016930170).}

\date{\today}

\begin{document}

\begin{abstract}
We give a structure theorem for projective manifolds $W_0$ with the property of admitting a 1-parameter deformation
where $W_t$ is a hypersurface in a projective smooth manifold $Z_t$.

Their structure is the one of special iterated univariate coverings which we call of normal type,
which essentially means that the line bundles where the univariate coverings live are tensor powers
of the normal bundle to the image $X$ of $W_0$.

We give applications to the case where $Z_t$ is projective space, respectively an Abelian variety.
\end{abstract}
\maketitle

{\em   Dedicated to Edoardo Sernesi  on the occasion of his 70th birthday}

%%%%%%%% pag  2 inizio %%%%%%%%%%%%%%%%%%

%%%%%%%%%%%%%%%%%%%%%%%%%%%
%%%%%%%%%%%%%%%%%%%%%%%%%%%%%%%%%%%%%%%%%%%%%%%%

\section {Introduction}

 Many years ago  Sernesi \cite{smalldef} showed that small deformations of complete intersections in projective space, of dimension $n \geq 2$ (the case of curves,
 $n=1$ is of quite different nature, see e.g. \cite{curvsurf}), are again
 complete intersections, unless the complete intersection defines a K3 surface (i.e.,  $n=2$ and the canonical bundle is trivial).
 Hence, in particular, smooth hypersurfaces in projective space $\PP^{n+1}$ form an open set in the Kuranishi
 space, respectively  an open set in the
 moduli space when they are of general type, unless $n=2$ and the degree equals $4$.
 In considering the closure of this set in the moduli space, we have to  deal  with varieties $W_0$ of the same dimension,  given together with a 
  generically finite  rational map $ \phi_0: W_0 \dasharrow \PP^{n+1}$.
 
 As shown by Horikawa in \cite{quintics}, already in the easiest nontrivial case $n=2$, $deg (W_0) =5$ the situation becomes rather complicated. But we show here
 that things are  simpler  in the case where $\phi_0$ is a morphism.

 A similar result to Sernesi's  holds for hypersurfaces  in an Abelian variety (Kodaira and Spencer's theorem 14.4 in \cite{KSII}), and we can consider the closure
 of the locus of hypersurfaces $X$ in Abelian varieties (for $n \geq 2$  the Abelian variety is just the Albanese variety of $X$) observing that  in this case  any limit $W_0$   has a generically finite Albanese map  $\phi_0 : W_0 \ra A_0$ (see for instance lemma 149 of \cite{C-top}).
 
 Also in this case we can ask the question of characterizing the morphisms  $\phi_0$ admitting a deformation which is a hypersurface embedding
 in some Abelian variety, deformation of the original one.

The main  motivation for posing  this question also in higher generality comes from the theory of  topological methods to moduli theory,
cf. \cite{C-top}; and, more specifically,  the theory of Inoue-type varieties, introduced in \cite{BC-IT}. 
In the theory of Inoue-type varieties, one can describe their moduli spaces explicitly  in the case where 
the morphism  $\phi_0$ has necessarily degree one onto its image. This is however a big restriction, and one would like to
consider also the case where the morphism  $\phi_0$ has degree at least two. We strive therefore
towards   a theory of multiple Inoue-type varieties and, in order to do this,
 we restrict ourselves in this paper to the special case where $\phi_0$ is a morphism which is generically finite onto its image, and
 where  the canonical divisor of  $W_0$ is ample.

 To  illustrate our main theorem, let us consider two simple examples,   the first one where the image of $W_0$ is the smooth
  hypersurface $X : = \{ \s = 0\} \subset \PP^{n+1}$,
 $\s$ being a homogeneous polynomial of degree $d$. We let  then $W_0$ be the complete intersection in the weighted projective space
 $\PP (1,1, \dots, 1, d)$ defined by the equations
 \begin{multline}  W_0 = \{ (z_0, z_1, \dots, z_{n+1}, w)| \s(z_0, z_1, \dots, z_{n+1}) = 0, \\ 
 P (z_0, z_1, \dots, z_{n+1}, w): = w^m + \sum_{i=1}^m w^{m-i} a_i(z_0, z_1, \dots, z_{n+1})=0\}.\end{multline}

 We can easily deform the complete intersection by deforming the degree $d$ equation adding a constant times the variable $w$,
 hence obtaining the  following complete intersection:

$$P (z_0, z_1, \dots, z_{n+1}, w)= 0,  t w -  \s(z_0, z_1, \dots, z_{n+1}) = 0 , t \in \CC.$$

Clearly, for $t=0$ we obtain the previous $W_0$, a degree $m$ covering of the hypersurface $X = \{ \s = 0 \}$, whereas for $t\neq0$ we can eliminate
the variable $w$ and obtain a hypersurface $W_t$ in $ \PP^{n+1}$ with equation (of degree $md$)

$$P (z_0, z_1, \dots, z_{n+1}, \s (z)/t)= 0.$$

\begin{example} {\bf (Iterated weighted  deformations)}.\label{weighted}

Now, one can iterate this process, and consider, in the weighted projective space
$$\PP (1,1, \dots, 1, d, dm_1, \dots, d m_k), \ \ m_1 | m_2 | \dots | m_k,$$
a complete intersection $W$ of multidegrees $( d, dm_1, \dots, d m_k, dm)$,
 
where $ m_k | m = : m_{k+1}$.

Then, necessarily, there exist constants $t_0, t_1, \dots, t_k$ such that the equations of $W$ have the following form,
where the $Q_j$'s are general weighted homogeneous polynomials of degree $ = d m_j$ (in particular we assume them
to be monic, so that the rational map to projective space is a morphism):

\begin{equation}\label{steq}
\begin{cases}
\s(z )=w_0t_0\\
Q_1(w_0, z)=w_1t_1 \\
\cdots\, \,\,\,\, \cdots \\
Q_{k}(w_0, \ldots, w_{k-1},z )=w_{k}t_k \\
Q_{k+1} (w_0, \ldots, w_k, z)=0.
\end{cases}
\end{equation}

Again, if all the $t_j$'s are $\neq 0$, we can eliminate the variables $w_j$, and we obtain a hypersurface
$ \{ F(z)=0 \}$ in $\PP^{n+1}$.

\end{example}

We claim that the above description generalizes, and
  the main idea  of the following main theorem is that one can replace weighted projective space
$\PP (1,1, \dots, 1, d, dm_1, \dots, d m_k), \ \ m_1 | m_2 | \dots | m_k,$ by the total space of a
direct sum of  line bundles over some  projective variety $X$,  or $Z_0 \supset X$, or over a family $\sZ$ 
of projective varieties.

The first assertion of the main theorem is that, in order that $\phi_0 : W_0 \ra Z_0$ deforms 
to a hypersurface embedding, a necessary condition is that $\phi_0 : W_0 \ra X : = \phi_0(W_0)$
is a smooth  iterated univariate covering of normal type (see the next section for this very restrictive condition).

The main theorem also gives sufficient conditions, given such a covering, for the existence 
of a deformation to a hypersurface embedding.

We then give the proof, and in the final section, we discuss the first applications to the case where the
target manifold $Z_0$ is projective space or an Abelian variety.

We defer the applications to the theory of multiple Inoue type varieties to a future paper.

We work of course over the complex numbers, and in several situations we consider also 
more general compact complex manifolds than projective manifolds.

%%%%%%%%%%%%%%%%%%%%%%%%%%%
\section{Statement of the main theorem}

 To give a clear statement of our results, we need to introduce the following terminology.

\begin{definition}
\hfill\break
i) Given a complex space (or a scheme) $X$, a {\bf univariate covering} of $X$ is a hypersurface $Y$, contained in a line bundle
over $X$, and defined there as the zero set of a monic polynomial. 

This means, $ Y = \underline{Spec} (\sR)$, where $\sR$ is the quotient algebra of the symmetric algebra 
over an invertible sheaf $\sL$, $ Sym (\sL) = \oplus_{i \geq 0} \sL^{\otimes i}  $, by a monic (univariate) polynomial:
$$ \sR : = Sym (\sL) / (P) , P = w^m+ a_1(x) w^{m-1} +a_2(x) w^{m-2}+\cdots +a_m(x).$$
Here $a_j \in H^0 (X, \sL^{\otimes j})$.
 The univariate covering is said to be {\bf smooth} if both $X$ and $Y$ are smooth.

ii) An {\bf   iterated univariate covering} $ W \ra X$ is a composition of  univariate coverings
$$f_{k+1} : W \ra X_{k}, f_k : X_k \ra X_{k-1}, \dots , f_1 : X_1 \ra X ,$$ whose associated line bundles are
denoted $\sL_k,  \sL_{k-1},  \dots, \sL_1,   \sL_0 $.

iii) In the case where $ X \subset Z$ is a (smooth) hypersurface, we say that the iterated univariate covering
is of {\bf normal type} if 

\begin{itemize}
\item
all the line bundles $\sL_j$ are pull back  from $X$ of a line bundle
of the form $\hol_X ( m_j X)$, and moreover
\item
$ m_1 | m_2 | \dots | m_k,$ and the degree of $f_j$ equals  $\frac{m_j}{m_{j-1}}$.
\item
we say that the iterated covering is {\bf normally induced} if moreover 
  all the coefficients $a_I(x)$ of the polynomials
  
$$ Q_{j}(w_0, \ldots, w_{j-1},x)= \sum_I  a_I(x) w^I $$

describing the intermediate extensions are sections of a line bundle $\hol_X(r(I)X)$ coming from
$H^0 (Z, \hol_Z(r(I)X))$.
\end{itemize}
\end{definition}  

\begin{remark}
The property that the iterated univariate covering $ W \ra X$ is normally induced clearly means that
it is the restriction to $X$ of an iterated univariate covering of $Z$.

The property that the former is smooth does not necessarily imply that also the latter is smooth.
\end{remark}

\begin{definition}
{\bf A 1-parameter deformation to hypersurface embedding} consists of the following data:

\begin{enumerate}
\item
 a one dimensional family of smooth projective varieties of dimension $n$ (i.e., a smooth
 projective holomorphic map $p :  \sW \ra T$ where $T$ is   a germ  of a smooth
holomorphic  curve at a point $0 \in T$)
mapping to another  family $\pi: \sZ  \ra T$ of smooth projective varieties of dimension $n+1$ via  a relative map $\Phi : \sW \ra \sZ$  such that $\pi \circ \Phi = p$ (hence we have the following commutative diagram)
\[\xymatrix@C-1pc{
  \sW \ar[rr]^{\Phi} \ar[dr]_{p} && \sZ \ar[dl]^{\pi} \\
  & T,
}\]
such that moreover 
\item

for $t\ne 0$ in $T$, $\Phi_t$ is an embedding of $W_t : = p^{-1}(t)$ in $Z_t$,
\item the restriction of the map $\Phi$ on $W_{0}$ is a generically finite morphism of degree $m$,   so that the image of $\Phi|_{W_{0}}$ is the cycle $\Si_{0}:=mX$ where  $X$ is a reduced  hypersurface in $Z_{0}$, defined by an equation
$X =\{\s=0\}$.
\end{enumerate}

\end{definition}

Put in  concrete terms, one can  take a local coordinate $t$ for $T$ at $0$, and
write,  locally around  $ \{ t=0\}$  the equation of the image $\Sigma : = \Phi (\sW)$ in $\sZ$  via  the Taylor series development in $t$, in terms of  local coordinates $z=(z_1, \ldots, z_n,   z_{n+1})$ on $Z_{0}$,
\[ \Si (z,t) : \s(z)^m+t\s_1(z)+t^2\s_2(z)+\cdots +t^{m-1}\s_{m-1}(z)+\ldots=0.\]

$\sW$ is a resolution of $\Si$ and the next theorems indicates exactly the sequence of blow-ups
needed in order to obtain the resolution $\sW$.

\begin{theorem} \label{MT} (A) Suppose we have a 1-parameter deformation to hypersurface embedding \[\xymatrix@C-1pc{
  \sW \ar[rr]^{\Phi} \ar[dr]_{p} && \sZ \ar[dl]^{\pi} \\
  & T.
}\]
  
and assume  that $K_{W_{0}}$ is ample.

Then we have:

(A1) $X$ is smooth,

(A2)  There are line bundles $\sL_0, \ldots, \sL_k$ on $\sZ$, such that  $\sL_j | _{Z_{0}} = \hol_{Z_{0}} ( m_j X)$ for $j=0, \ldots, k$,  with
  $ 1 = m_0 | m_1 | m_2  \dots | m_k | m_{k+1} := m $  (here $m$ is the degree of the morphism $\Phi_0 : W_0 \ra X$),    and such that 
  $W_{0}$ is  a complete intersection in $\sL_0\oplus\cdots\oplus\sL_{k}| _{ Z_{0} }$, with $\Phi_0$
  a normally induced   iterated  smooth univariate covering .

(A3)  $\sW$ is obtained  from  $\Sigma : = \Phi (\sW)$  by a finite sequence of blow-ups.

 Moreover  the local equations of $\sW$ are  of the  following standard form 
\begin{equation}\label{steq2}
\begin{cases}
\s(z)=w_0t\\
Q_1(w_0, z)=w_1t \\
\cdots\, \,\,\,\, \cdots \\
Q_{k}(w_0, \ldots, w_{k-1},z)=w_{k}t\\
Q_{k+1}(w_0, \ldots, w_k, z, t)=0.
\end{cases}
\end{equation} 
(B1) Conversely, take any  smooth iterated univariate covering of normal type   
$$\phi_0 : W_{0} \ra  X \subset Z_0$$
and take any 1-parameter family $\sZ$ of deformations of $Z_{0}$.

 Then the  line bundle $ \hol_{Z_{0}} ( X )$ extends to a line bundle
$\sL_0 $ on the whole family $\sZ$.   
And $W_{0}$ deforms to a hypersurface embedding
if, for all  $ i \geq 2$,  every section in $H^0 (Z_{0}, \hol_{Z_{0}} (iX))$ and every section in 
$H^0 (X, \hol_{X} (iX))$  comes from a section  in $H^0 (\sZ,\sL_0^{\otimes i})$.

 (B2) This holds in particular, when  the family $\sZ$ is trivial, $\sZ \cong Z_{0} \times T$, 
  if the necessary condition of being normally induced is fulfilled.

\end{theorem}

\begin{remark}\label{precisions}
(b1) More precisely, in (B1) above, there is a 
family $\sW$  such that $\sW $ is a complete intersection in $\sL_0\oplus\cdots\oplus\sL_{k}$
($\sL_i = \sL_0^{\otimes {m_i}}), $
 and $\sW $  is given  as above; moreover, for $t\ne 0$ in $T$, 
 the morphism $\Phi_t$, induced on $W_{t}$ by the bundle projection on $Z_{t}$,  is an embedding.
 
(b2)    sufficient condition in (B2) is the surjectivity  of  
$H^0 (Z_{0},\hol_{Z_{0} }(iX)) \ra H^0 (X, \hol_{X} (iX))$ for $ i\geq 2$; 
this is implied by  $H^1 (Z_{0},\hol_{Z_{0} }(iX)) = 0, \forall i \geq 1.$
\end{remark}
 
\begin{remark}\label{extends}
The line bundle $ \hol_{Z_{0}} (X)$ extends to a line bundle
$\sL_0 $ on the whole family $\sZ$, because of the Lefschetz (1,1) theorem, since $ \hol_{Z_{0}} (mX)$ does.

Observe moreover that there is a (non-canonical) isomorphism $$\Pic^0 (Z_{0}) \cong \Pic^0 (Z_{t}),$$
whereas  in general there is no isomorphism of  $\Pic (Z_{0})$ with  $ \Pic (Z_{t})$.
\end{remark}

\begin{remark}
(i) Consider the family of submanifolds of weighted projective spaces given in example \ref{weighted},
with equations \ref{steq}, and consider the 1-parameter deformation where we set $t_j = t^{n_j}$.

Then the equations considered in the proof of the theorem are many more than equations \ref{steq}, 
since for instance from the equation $\s(z) = t^{n_0} w_0$
we recover $w_0$ not directly but only after an iterated procedure: we inductively set $\s(z) = t^a v_{a-1}$,
so finally we get $w_0 = v_{n_0 -1}$.
\end{remark}

%%%%%%%%%%%%%%%%%%%%%%%%%%%%%
\section{Auxiliary results and proof of the main theorem}

The following  observation plays an important role in the proof.

 \begin{lemma} \label{smooth}
Suppose we have a one dimensional smooth family $p :  \sW \ra T$ of smooth projective varieties of dimension $n$ mapping to another  flat family $q: \sY  \ra T$ of  projective varieties of the same dimension via  a relative map $\Psi: \sW \ra \sY$ over a smooth
holomorphic  curve $T$ such that $q \circ \Psi = p$, i.e. we have the commutative diagram
\[\xymatrix@C-1pc{
  \sW \ar[rr]^{\Psi} \ar[dr]_{p} && \sY \ar[dl]^{q} \\
  & T.
}\]
  
Assume that
\begin{enumerate}
\item $\sY$ is normal and Gorenstein,
\item $\Psi$ is birational,
\item for $t\ne 0$ in $T$, $\Psi$ induces an isomorphism,
\item $K_{W_{0}}$ is ample.
\end{enumerate}
Then we have  that $\Psi$ is an isomorphism, in particular $W_{0} \cong Y_{0}$.
\end{lemma}

\begin{proof}

We have $K_{\sW}=\Psi^*(K_{\sY}) + \sB$. Since we assume that $\Psi$ induces an isomorphism for $t \ne 0$ in $T$, the support of 
the Cartier divisor $\sB$ is contained in $W_{0}$, which is  irreducible.

Now  $Y_{0}$ has dimension $n$ and the morphism  $\Psi_0 : W_{0} \ra Y_{0}$ is generically finite, hence we conclude that $\sB=0$.
In particular, $K_{\sW}=\Psi^*(K_{\sY})$; restricting to the special fibre, we obtain  $K_{W_{0}} =\Psi_0 ^*(K_{Y_{0}})$. Since 
by assumption $K_{W_{0}}$ is ample, $\Psi_0$ is finite, hence also $\Psi$ is finite, hence an isomorphism in view of the normality of $\sY$.

\end{proof}

\begin{remark} Without the  assumption that $K_{W_{0}}$ is ample one can only assert that $Y_{0}$ is normal with at most canonical singularities.
\end{remark}

\begin{lemma} \label{divide}
In the hypotheses of theorem~\ref{MT}, we have that $\sigma^{m-i} \, | \, \sigma_i$ for $i=1,\ldots, m-1$.
\end{lemma}

\begin{proof}
Since the map $W_0 \to X$ is a generically finite map of degree $m$, given  a general point $p$ of $X$, the inverse image of $p$ consists of $m$ points $q_1, \dots, q_m$,
and at each $q_i$  the rank of the derivative of the morphism $W_0 \to Z_0$ is equal to $n$. Hence we get local coordinates $(w_1, \ldots, w_n, t)$ for $\sW$ at $q_i$ and local coordinates $(z_1, \ldots, z_n, z_{n+1},  t)$ for $\sZ$ at $p$ (not depending on $i$)
such that $\Phi$ is given by a function $f_i$
\[ f_i(w_1, \ldots, w_n, t)=(w_1, \ldots, w_n, \varphi_i(w_1, \ldots, w_n, t), t)\]
such that 
\[ f_i(w_1, \ldots, w_n, 0)=(w_1, \ldots, w_n, 0, 0).\]
Here  $\varphi_i (w_1, \ldots, w_n, t)$ is a holomorphic function. 

Hence at $p\in X$, there are variables $(z_1, \ldots, z_n, z_{n+1}, t)$ such that $\s =  z_{n+1}$,
and $\Si$ consists of $m$ smooth branches with equation 
$$\s - \varphi_i(z_1, \ldots, z_n, t) =0, \s = z_{n+1}$$ for $i=1, \ldots, m$.

Hence, setting $z: =(z_1, \ldots, z_n)$ the local equation of $\Si$ is 
\begin{equation} \label{Planecurve}\prod_{i=1}^m(\s-\fie_i(z, t))=\s^m-a_1(z, t)\s^{m-1}+\ldots+a_m(z, t) \end{equation}
and $t^i | a_i$, since $t |\fie_i$.
\end{proof}

\begin{remark}
 Since  we assume that $\Phi_t$ is an embedding when $t\ne 0$, for a fixed value of  $z: =(z_1, \ldots, z_n)$  
 equation \ref{Planecurve} yields a 
 plane curve  with $m$ smooth branches. Equivalently, we may view equation \ref{Planecurve} as giving
 a plane curve over a non algebraically closed field (the fraction field of the ring of power series in $z$).

\end{remark}

\begin{lemma} \label{local}
Assume the same hypotheses  of part (A)  of theorem \ref{MT} with $n=1$. Let $p\in W_0$ be any point. 
Then there is a neighborhood $U$ of $p$ in $W_0$ such that $\phi_0(U)$ is a smooth curve of $Z_0$ at $\phi_0(p)$.
\end{lemma}

\begin{proof} We have a factorization $\phi_0=\nu\circ \psi_0$ near the point $p$
\[ W_0\stackrel{\psi_0}{\longrightarrow} X^{\rm nor}\stackrel{\nu}{\longrightarrow} X\]
where $\nu: X^{\rm nor}\to X$ is the normalization. 

There are respective local coordinates $w$ around $p\in W_0$, $u$   around  $\psi_0(p)$, $x,y$ around $ \phi_0(p)$ such that:
\begin{itemize}
\item
$\psi_0(w)=w^h=u$,
\item
 $x(u), y(u)$ are a Puiseux parametrization of the branch associated to $\phi_0 ( p)$.
 If the branch is nonsingular, without loss of generality  $y\equiv 0$ and there is nothing to prove.
 \item
 Otherwise   the branch is singular, and we have a Puiseux parametrization of the form
\[ \begin{cases} 
x=u^d\\
y=u^e+g(u)u^{e+1}, \,{\rm where}\, \,\, e>d.
\end{cases}\]
\end{itemize}
Hence the branch is  the zero set of a pseudo-polynomial in $x$,
\[x^d+a_1(y)x^{d-1}+\cdots +a_d(y).\]

Now, we can write, locally  identifying $\sZ$ to $\CC^2\times T$,   
$\Phi (w,t) $ as $$ \Phi (w,t)= (\phi(w,t),t)$$ and $\phi(w, t)$ as 
\[\begin{cases}
x=w^{dh}+t\phi_1(w, t) \\
y=w^{eh}(1+g(w^h)w^h)+t\phi_2(w, t)
\end{cases}\]
The link of the branch $\phi_0(p)$ is, by Zariski's theorem \cite{Zar}, an iterated non-trivial toral knot, running $h$ times. However, $L_0$ is isotopic to $L_t$, which is  gotten by  the image of the circle $|w|=\epsilon$ under  $\phi(w, t)$. 

Observe that the map $\phi(w,t)$, by purity of branch locus, ramifies on a curve $R$, which is not contained in $\{ t=0\}$:
since $W_t$ is embedded for $t\neq 0$, it follows that $R$ is exceptional and that the curves $W_t$ are a family of curves through the origin
$x=y= 0 \in \CC^2$. Moreover, again since $W_t$ embeds, the reduced curve $R_{red}$ is a smooth curve which projects isomorphically to the $t$-axis,
so we may assume without loss of generality that $R_{red} = \{ w=0\}$.

 The link $L_t$ is contained in the submanifold $W_t$  of a four dimensional ball $B$ around the origin; $W_t$ is a smooth holomorphic curve 
through  the point $\phi(0, t) = 0$,
   hence $L_t$ yields  an unknotted circle  $S^1\subset S^3$.  We have derived a contradiction from assuming $d>1$, and that the branch is singular.

\end{proof}

\begin{lemma} \label{coincide}
Assume the same hypotheses of Theorem~\ref{MT}.  Let 
$$W_0^f=\{ p\in W_0\, | \, \text{$\phi_0$ is finite at $p$}\}.$$ Then if $p\ne p'\in W_0^f$ and $\phi_0(p)=\phi_0(p')=y_0$, then the germs of analytic subsets  $Y_0 : = \phi_0(U_p)$ and $Y'_0 : = \phi_0(U_{p'})$ coincide.
\end{lemma}

\begin{proof} Cut now $\CC^{n+1}$ with a general linear $\CC^2$ passing through $y_0$, so that we get curves $\CC^2\cap Y_0=C$, $\CC^2\cap Y_0'=C'$, passing through $y_0$. Now $(C\cdot C')_{y_0}=d\ge 1$, and indeed $d$ is a topological invariant, it is the linking number of $C\cap S^3$ and $C'\cap S^3$.

For $|t|<<1$,  we  can deform  the curves $C, C'$ which are in the image of $W_0$, to  curves  $C(t), C'(t)$ in the image
of $W_t$: these have the property that  in a neighborhood of $y_0\in\CC^2$ they intersect in $d$ points, counted with multiplicities. Hence, if $y(t)\in C(t)\cap C'(t)$, then there are $p(t)\ne p'(t)\in W_t$ with $\phi_0(p(t))=\phi_0(p'(t))$. This is a contradiction.

\end{proof}

\begin{proposition} \label{normal}
Assume the same hypotheses of Theorem~\ref{MT}.  Then $X$ is normal.
\end{proposition}

\begin{proof} 
$X$ is a hypersurface in a smooth manifold, hence it is normal if and only if it is smooth outside of codimension $2$
in $X$.

The image of the points in $W_0 \setminus W_0^f$ is a Zariski closed subset of codimension$\ge 2$ in $X$.

Hence it suffices to consider points $q \in X$ which are  image points only of points $p \in W_0^f$.

By  Lemma~\ref{coincide} the germ of $X$ at $q$ equals the image of the germ  of $W_0$ at $p$.

Let $\sC_k=\{ p\in W_0\, |\, {\rm rank}(D\phi_0)_p=k\}$. Then $\phi_0(\sC_k)$ has dimension $\le n-2$ for $k\le n-2$, and 
we conclude that  $X$ is  smooth outside of codimension $2$ unless ${\rm dim}\,\phi_0(\sC_{n-1})=n-1$.

 It also suffices to consider the general point $p$ of $\sC:=\sC_{n-1}$, where $\sC$ is smooth 
 of  dimension $n-1$ and rank$(D(\phi_0|_{\sC}))= n-1$.

Let $(\hat d-1)$ be the multiplicity of $\sC$ in the  locus given by the $n \times n$
minors of the derivative matrix. 

There are local coordinates $(v, w)$ in a neighborhood of $p \in W_0$, with $v=(v_1,\ldots, v_{n-1})$,  such that 
$\sC=\{ w =0\}$ and  such that
\[ \phi_0(v, w)=(v, x(v, w), y(v, w)).\]
The locus  given by the $n \times n$
minors of the derivative matrix is then just the locus $\frac{\partial x}{\partial w}=0, \frac{\partial y}{\partial w}=0$. 

Without loss of generality, at the general point of $\sC$ we can assume that 
\[ x(v, w)=\sum_{i=\hat d}^{\infty} a_i(v)w^i,\]
and, since $a_{\hat d}(x)\not\equiv 0$, we may assume that $a_{\hat d}(0)\ne 0$. Hence we choose coordinates $x,y$ with
\[ \begin{cases}
x=w^{\hat d},\\
y=w^e+\cdots \,\, \text{with $e>\hat d$ and $\hat d\not\vert e$}.
\end{cases} \]

Then, if $y\not\equiv 0$, we get, for any $v$ in a neighborhood of $0$, a singular curve branch. The same argument as for the case $n=1$ applied to $x(0, w), y(0, w)$ gives a contradiction. So we have established the proof.

\end{proof}

\begin{lemma} \label{curve}
Let $C$ be a germ of a plane curve singularity consisting of $m$ smooth branches with non vertical tangents, i.e. the local equation of $C$ is
$F(y, t) = 0$, with 
\[\begin{array}{l}
F(y, t):=\\
=\prod_{i=1}^m (y-\fie_i(t)) \\
=y^m-\s_1(\fie_1, \ldots, \fie_m)y^{m-1}+\s_2(\fie_1, \ldots, \fie_m)y^{m-2}+\cdots +(-1)^m\fie_1\cdots\fie_m.
\end{array}\]
Here, $\s_i$ is the i-th elementary symmetric function, and  $t| \fie_i(t)$, since $(0,0)$ is the singular point, hence $t^i$ divides 
$\s_i(\fie_1, \ldots, \fie_m)$.

 Then the singularity can be resolved by iterated blow-ups of the form:
\[\begin{cases}
y=w_0t\\
D_1(w_0)=w_1t \\
D_2(w_1)=w_2t \\
\cdots\, \,\,\,\, \cdots \\
D_k(w_{k-1})=w_kt
\end{cases}\]
where the $D_j$'s are monic polynomials.
\end{lemma}

\begin{proof} 
$F(y,t)$ is a pseudo-polynomial. Write
\[F(y, t)=y^m+t \  b_1(t)y^{m-1}+\cdots+t^mb_m(t).\]
The first blow-up yields
\[ P(w_0, t)=w_0^m+b_1(t)w_0^{m-1}+\cdots+b_m(t).\]
Let $P(w_0, t)=P_0(w_0)+tP_1(w_0)+t^2P_2(w_0)+\cdots$
where $P_0(w_0)$ is a monic polynomial of degree $m$ and ${\rm deg}_{w_0}P_j(w_0)\le m-1$ for $j\ge 1$. Looking at the two partial derivatives for $t=0$, the proper transform $C_0$ is smooth if and only if
\[P_0(w_0)=0, \,\frac{\pa}{\pa w_0}P_0(w_0)=0, \, P_1(w_0)=0\] have no common roots. 

If not, let $D_1(w_0)={\rm gcd}$ of the above three polynomials, so that $P_0(w_0)=D_1^r\cdot G$ for some $2\le r\le m$ such that $ r \cdot deg(D_1) \leq m$. 

To continue, observe that $D_1$ is again monic, so there is no tangent $t=0$, and   it suffices to blow-up $t=D_1=0$ setting  $D_1(w_0)=w_1t$.

We get equations
\[ \begin{array}{l} 
D_1(w_0)^r\cdot G(w_0) +tD_1(w_0)\tilde P(w_0)+t^2\cdots \\
= t^r [ w_1^r\cdot G(w_0)+w_1^{r-1} \cdots ].
\end{array}\]
 where the divisibility of the second term by $w_1^{r-1}$, and similarly for the next terms follows
  by applying once more Lemma~\ref{divide}.

 We continue this process until all branches are separated.
 
 \end{proof}

\noindent{\bf Proof of Theorem \ref{MT}}

Recall the Taylor series development in $t$ of the equation in $\sZ$ of the image $\Sigma : = \Phi (\sW)$: 
\[ \Si (z,t) : = \s(z)^m+t\s_1(z)+t^2\s_2(z)+\cdots +t^{m-1}\s_{m-1}(z)+\ldots.\]

Choose a general point $p$ in $X$.  In order to show that the process terminates
we shall at a later moment consider a germ of plane curve $C$ passing  through $p$, consisting of $m$ smooth branches with non vertical tangents, and obtained as  a linear section of $\Si$ (i.e., $C$ is obtained by setting 
$(z_1, \dots, z_n)$ = constant in appropriate local coordinates). 

Observe that, by Lemma~\ref{divide}, $\s^i | \s_{m-i}$ for $i=1, \ldots, m-1$. So we can rewrite the equation of $\Si$ as follows:
\begin{equation} \label{Sigma}
\Si (z,t) : = \s(z)^m+a_1(z, t)\s(z)^{m-1}+a_2(z, t)\s(z)^{m-2}+\cdots +a_m(z, t)=0
\end{equation}
where $t^i | a_i(z, t)$.

Please observe that the above equation is not just a local equation, but that it is a global equation for a section
of a line bundle on $\sZ$ (see remark \ref{extends}) $$\hol_{\sZ}(\Sigma) = \sL_0^{\otimes m}.$$

Now, by setting $\s=tw_0$ and $a_i(z, t)=t^ib_i(z, t)$ in  equation~\ref{Sigma}, we obtain the equation 
\[ P_{w_0}(w_0, z, t):=w_0^m+b_1(z, t)w_0^{m-1}+\cdots+ b_m(z, t)=0.\]
This is a hypersurface and its singular locus is contained in $t=0$ by our assumption. Write
\[ b_j(z, t)=b_{j, 0}(z)+tb_{j, 1}(z)+\cdots,\]
so that the $b_{j, 0}(z)$'s are sections on $Z_0$ of a line bundle of the form $\hol_{Z_0} (i X)$
(this observation shall be repeated in the sequel, leading to the proof that we get a normally induced
covering of $X$).

Hence we can  write
\[ P(w_0, z):=w_0^m+b_{1, 0}(z)w_0^{m-1}+\cdots+b_{m, 0}(z).\]
Hence  the equation $P_{w_0}$ has the development  
\[ P_{w_0}(w_0, z, t) = P(w_0, z)+t\sum_{j=1}^mb_{j, 1}w_0^{m-j}+t^2\cdots.\]
Set $\hat P(w_0, z): =\sum_{j=1}^mb_{j, 1}w_0^{m-j}$;  then the gradient  of $P_{w_0}$ for $t=0$ equals 
\[ ( \frac{\pa P}{\pa w_0},\,\hat P, \, \frac{\pa P}{\pa z_1}, \, \frac{\pa P}{\pa z_2},\ldots ).\]

The singular locus is contained in $t=0$ and is there given by
\[ \begin{cases}
P(w_0, z)=0\\
\frac{\pa}{\pa w_0} P(w_0, z)=0 \\
\hat P(w_0, z)=0\\
\frac{\pa}{\pa z_i}P(w_0, z)=0 \,\,({\rm for}\, i=1, \ldots, n).
\end{cases} \]

The hypersurface defined by $P_{w_0}(w_0, z, t)=0$ is normal unless it is singular in codimension 1.
 Assume that  the hypersurface defined by $P_{w_0}(w_0, z, t)=0$ is not normal, so that  it is singular in codimension 1. 
Observe that the square free part of $P(w_0, z)$ is irreducible because its zero set   is the image of an irreducible variety $W_0$,
hence for $t=0$ the vanishing of  $P(w_0, z)$ should imply the vanishing of the other polynomials.

Since  $P(w_0, z) = 0  \Rightarrow \pa P(w_0, z)/\pa w_0=0$,  hence  $P(w_0, z)=Q(w_0, z)^r,$
for some $ 2\le r\le m$. 

And $Q(w_0, z)$ is again irreducible and $Q(w_0, z) | \hat P(w_0, z)$. Now, since $r\ge 2$,
\[ \pa P(w_0, z)/\pa z_i=rQ(w_0, z)^{r-1}\pa Q(w_0, z)/\pa z_i,\]
so the last conditions are automatically fulfilled.

If we write the equation $P_{w_0}(w_0, z, t)$ in terms of $Q(w_0, z)$, then
\[Q(w_0, z)^r+tQ(w_0, z)\tilde P(w_0, z, t)+t^2\cdots.\]
Then again by Lemma~\ref{divide}, $Q(w_0, z)^{r-2} |\tilde P(w, z, t)$, and so on.

Observe that $Q(w_0, z)=0$ gives a covering $W'_0$ of $X$ of degree $m/r$.

In view of Lemma~\ref{curve}, we set $Q(w_0, z)=w_1t$. Then we get an iterated covering as in Lemma~\ref{curve}, where $P=D_1^r=Q^r$. We consider   now  a germ of plane curve $C$ through the general point $p\in X$ consisting of $m$ smooth branches with non vertical tangents by taking linear  sections.
Since the resolution of $C$ is obtained by a finite sequence  of blow-ups, we get an iterated covering as in Lemma~\ref{curve}, which is normal and a complete intersection. 

 Lemma~\ref{smooth} implies that we have then obtained $\sW$. 
 
 Since $W_{0}$ is smooth, let us set $t=0$ in the above equations \ref{steq2}, and observe that the matrix of derivatives has
 triangular form. So, smoothness of $W_{0}$  implies the smoothness of $X$
 and of all the intermediate coverings.
 
 Therefore we have shown (A1), (A2) and (A3).

 Let us show now the converse,  (B1) and (B2).
 
Assume we  are given an  iterated smooth univariate covering $W_0$  of $X$ which is normally induced, defined by equations:

\begin{equation}\label{steq3}
\begin{cases}
\s(z)= 0 \\
Q_1(w_0, z)= 0 \\
\cdots\, \,\,\,\, \cdots \\
Q_{k}(w_0, \ldots, w_{k-1}, z)= 0\\
Q_{k+1}(w_0, \ldots, w_k, z)=0.
\end{cases}
\end{equation}

Claim I): if $H^0 (Z_{0}, \hol_{Z_{0}} (iX))$ surjects onto  
$H^0 (X, \hol_{X} (iX))$ for each $i \geq 2$, then every iterated  univariate covering $W_0$  of $X$
is normally induced, i.e., it extends to an iterated  univariate covering of $Z_0$.

\Proof
First of all we can put the equations of the iterated covering $W_0$  of $X$ in Tschirnhausen form.
 Here the polynomial equation of a univariate covering is said to be in Tschirnhausen form if $a_1(x) \equiv 0$,
and  every covering can be put in 
Tschirnhausen form after an automorphism replacing $w$ with $ w - \frac{1}{m}  a_1 (x)$.

Second,  the coefficients $a_{j,I} $ of the polynomials $Q_j$ are now given by sections of 
line bundles of the form $ \hol_{X} ( n(I,j) X )$, where $n(I,j) \geq 2$. By assumption, they extends to sections
of $ \hol_{Z_0} ( n(I,j) X )$.

That this holds if $H^1 (Z_{0}, \hol_{Z_{0} }(iX)) = 0, \forall i \geq 1,$
follows immediately from the long exact cohomology sequence associated to the exact sequence
$$ 0 \ra  \hol_{Z_{0} }(i X) \ra \hol_{Z_{0} }((i+1)X) \ra \hol_{X }((i+1)X) \ra 0  ,$$
for $i \geq 1$.

\qed

Let's pass to the proof of (B1).

 By our assumption, $\s (z)$ extends to a section $\s(z,t)$ of the line bundle
$\sL_0 $ on the whole family $\sZ$. Similarly  the sections $a_{j,I} $ extend to sections $a_{j,I} (z,t)$ over $\sZ$,
hence we can also extend the polynomials $Q_{j}(w_0, \ldots, w_{j-1}, z)$ to polynomials 
$Q_{j}(w_0, \ldots, w_{j-1}, z,t)$.

Then we define the iterated univariate covering $\sW$ of $\sZ$ via the following equations:

\begin{equation}\label{steq4}
\begin{cases}
\s(z,t)=w_0t\\
Q_1(w_0, z, t)=w_1t \\
\cdots\, \,\,\,\, \cdots \\
Q_{k}(w_0, \ldots, w_{k-1},z,t )=w_{k}t\\
Q_{k+1}(w_0, \ldots, w_k, z, t)=0.
\end{cases}
\end{equation}

To finish the proof of (B2), observe that  if the family $\sZ$  is trivial, $\sZ = Z_{0} \times T$,  then 
obviously 
$$H^0 (Z_{0},\hol_{Z_{0} }(iX)) \subset H^0 (\sZ,\hol_{\sZ}(iX)),$$
hence there is no problem to extend the iterated univariate covering to one of $\sZ$.

\qed

%%%%%%%%%%%%%%%%%%%%
\section{Applications}

The first applications that we shall give are, more or less, direct corollaries of the previous general results.

\subsection{Hypersurfaces in projective space}
\begin{corollary}
The   smooth manifolds $W_0$ with ample canonical divisor which admit a 1-parameter deformation 
to a hypersurface embedding
$$p: \sW \ra T, \ \ \Phi : \sW \ra \PP^{n+1} \times T$$ (here, for $t \neq 0$, $W_t$ is a smooth hypersurface in $\PP^{n+1}$)
are exactly the  iterated weighted deformations of  example \ref{weighted} with $m d > n+2$, and the family $\sW$ is a pull back from the
family  in \ref{steq}.

The class of such manifolds $W_0$ is open in the Kuranishi space for $ n \geq 2$.
\end{corollary}
\Proof
Either $W_0$ is a hypersurface, and there is nothing to prove (set $m=1$),
or the degree of $ \phi_0 : W_0 \ra X$ is $ m \geq 2$.

Then theorem \ref{MT} applies, and it is easy to see that we get a manifold in the family \ref{steq}.
The converse is direct (set $t = t_j \ \forall j$  in the family \ref{steq}).

The proof of the second  statement follows imitating quite closely Sernesi's argument in \cite{smalldef}
for the case of complete intersections in weighted projective spaces.
\footnote{ In the more general situation of theorem \ref{MT} one can use Horikawa's theory of
deformations of finite holomorphic maps, see \cite{horikawaNB}, also the lecture notes \cite{montecatini}. }

\qed

\subsection{Hypersurfaces in an Abelian variety}

\begin{corollary}
The  smooth manifolds with ample canonical divisor $W_0$  which admit a 1-parameter deformation 
$p: \sW \ra T$ where, for $t \neq 0$, $W_t$ is a smooth hypersurface in an Abelian variety $A_t$,
are, for $n \geq 2$,  exactly the iterated smooth univariate coverings $W_0 \ra X$ of normal type,
where $X$  is an ample divisor in an Abelian variety $A_0$.
\end{corollary}
\Proof
$A_t$ is, for $t \neq 0$, the Albanese variety of $W_t$, since  $H^i(\hol_{A_t}(-W_t)) = 0$ for $i=1, 2$,
hence the exact cohomology sequence associated to the exact sequence 
$$ 0 \ra  \hol_{A_{t} }(-W_t) \ra \hol_{A_{t}} \ra \hol_{W_t } \ra 0  $$
yields $h^1(\hol_{A_t})= h^1(\hol_{W_t}) .$

We obtain therefore a morphism 
$$ \Phi : \sW \ra \sA, \ A_t = Alb (W_t)$$
 that induces a 1-parameter deformation to hypersurface embedding in Abelian varieties by lemma 149 of \cite{C-top}.

Our main theorem applies, in particular (B2) holds since we have 

$H^1(\hol_{A_0}( i X))= 0$
for all $ i \geq 1$.

\qed

\subsection{Multiple Inoue-type varieties}

We can now extend the definition of Inoue type variety (see \cite{BC-IT} and \cite{C-top})
in the following way:

\begin{definition}
A primary multiple Inoue-type variety $W_0$ of smoothing type is
a normally induced   smooth iterated  univariate covering $W_0 \ra X$, where
$X$ is a smooth ample subvariety of a projective classifying space $Z_0$.

A multiple Inoue-type variety $Y_0$ of smoothing type is a  quotient $W_0 /G$,
for the free action of a finite group, of a primary 
multiple Inoue-type variety $W_0$ of smoothing type.
\end{definition}

We do not give further applications here, hopefully in a future paper.

%%%%%%%%%%%%%%%%%%%%%%%%%%%%%%%%%%%%%%%%%%%%%%%%%%%%%%%%%%%%%%%%%%%%%%%%%%%%%%%%%%%%%%%%%%%%%%%%%%%%%%%%%%%%%%%%%%%%%%%%%%%%
%%%%%%%%%%%%%%%%%%%%%%%%%%%%%%%%%%%%%%%%%%
%%%%%%%%%%%%%%%%%%%%%%%%%%%%%%%%%%%%%%%%%%

\end{document}